\documentclass[a4paper, 12pt]{amsart}
\usepackage{amsmath,amssymb,amsthm}
\usepackage{amscd}
\usepackage{array,enumerate}
\newtheorem{Thm}{Theorem}[section]

\newtheorem{Lem}[Thm]{Lemma}

\newtheorem*{claim}{Claim}

\theoremstyle{definition}

\newcommand{\Cs}{C$^\ast$}

\newcommand{\id}{\mbox{\rm id}}

\newcommand{\rg}{\mathop{{\mathrm C}_{\mathrm r}^\ast}}

\newcommand{\rc}{\mathop{\rtimes _{\mathrm r}}}

\newcommand{\rca}[1]{\mathop{\rtimes _{{\mathrm r}, #1}}}

\newcommand{\IB}{\mathbb B}
\newcommand{\IC}{\mathbb C}

\newcommand{\IN}{\mathbb N}
\newcommand{\IQ}{\mathbb Q}
\newcommand{\IR}{\mathbb R}

\newcommand{\fH}{\mathfrak{H}}

\newcommand{\cG}{\mathcal G}

\newcommand{\cM}{\mathcal M}

\DeclareMathOperator{\bigfp}{\lower0.25ex\hbox{\LARGE $\ast$}}

\title[\Cs-simplicity has no local obstruction]
{\Cs-simplicity has no local obstruction}
\author{Yuhei Suzuki}
\address{Department of Mathematics, Faculty of Science, Hokkaido University,
Kita 10, Nishi 8, Kita-Ku, Sapporo, Hokkaido, 060-0810, Japan}
\email{yuhei@math.sci.hokudai.ac.jp}
\begin{document}
\maketitle
\begin{abstract}
In 2016, I solved a problem of de la Harpe in 2006: Is there a non-discrete \Cs-simple group? However the solution was not fully satisfactory as the provided \Cs-simple groups (and their operator algebras) are very close to discrete groups. All previously known examples are of this form. In this article I give yet another construction of non-discrete \Cs-simple groups. The statement in the title then follows. This in particular gives the first examples of non-elementary \Cs-simple groups (in Wesolek's sense).
\end{abstract}
\section{Introduction}
A fundamental motivation of operator algebra theory is
to give a framework to understand locally compact groups.
Successful achievements in this direction include Glimm's dichotomy theorem,
the Kasparov theory, the Baum--Connes theory, and Popa's deformation/rigidity theory. 
Also, via the (reduced) crossed product construction, locally compact groups produce interesting examples of concrete operator algebras.
On the one hand, for \emph{discrete} groups, many deep structural results
on these operator algebras have been established.
On the other hand, the \emph{non-discrete} counterpart is not yet on the comparable level.

In this article, we focus on \emph{\Cs-simplicity} \cite{Pow}, the simplicity of the reduced group \Cs-algebra,
of locally compact groups.
For discrete groups, satisfactory characterizations of \Cs-simplicity
were established in the last decade (see e.g., \cite{BKKO}, \cite{KK}).
However, the results do not (at least directly) extend to non-discrete groups.
\Cs-simplicity of non-discrete groups is still a mysterious property.
A main reason of the difficulty would be the lack of interesting examples.
Indeed even the existence of such a group, asked by de la Harpe \cite{Har},
was not known until \cite{Suzsim}.
While such groups are now found (\cite{Suzsim}, \cite{Rau2}),
all the currently known examples are
\emph{very close} to discrete groups.
(More precisely, they are essentially the projective limit of discrete groups
of particularly good form; see the Proposition in \cite{Suzsim}, which is the only previously known result to produce a non-discrete \Cs-simple group.
In particular, all these groups are \emph{elementary} in Wesolek's sense
by Theorem 3.18 of \cite{Wes}.)

In this article, we provide a new framework to
produce non-discrete \Cs-simple groups.
Note that \Cs-simple groups must be \emph{totally disconnected} by Theorem A of \cite{Rau}.
Thus our attention is naturally restricted to totally disconnected groups.
As a result of the new construction, we conclude
the statement in the title: every totally disconnected locally compact group
realizes as an open subgroup of a \Cs-simple group.
In particular we obtain the first examples of \Cs-simple groups which are \emph{non-elementary} in Wesolek's sense \cite{Wes}.
We believe that our new construction sheds new light on (non-discrete) \Cs-simplicity, and that our proof gives a new insight
on the analysis of the group and reduced crossed product operator algebras of non-discrete groups.
\section{Preliminaries}
Here we fix notations, and prove a basic lemma.
\subsection*{Notations}
Throughout the article, let $G$ be a totally disconnected locally compact group. 
We fix a left Haar measure $\mu$ on $G$.
Put $L^2(G) := L^2(G, \mu)$.
Let $\lambda \colon G \curvearrowright L^2(G)$ denote the left regular representation.
The representation $\lambda$ integrates to the $\ast$-representation of the group algebra
$C_c(G)$ on $L^2(G)$ which is given by the convolution product.
The reduced group \Cs-algebra $\rg(G)$ is the operator norm closure
of $C_c(G) \subset \IB(L^2(G))$.
For a compact open subgroup $K$ of $G$, set
\[p_K:= \mu(K)^{-1} \chi_K \in C_c(G) \subset \rg(G).\]
Observe that $p_K$ is the orthogonal projection onto the $K$-fixed point space $L^2(G)^K$. 
A theorem of van Dantzig \cite{Dan} shows that the set of compact open subgroups $K<G$ form a local basis at the identity element $e\in G$.
Hence the net $(p_K)_{K<G}$ forms an approximate unit of $\rg(G)$.
For a closed subgroup $H$ of $G$, we identify
$\rg(H)$ with the \Cs-subalgebra of the multiplier algebra $\cM(\rg(G))$ in the obvious way (cf.~\cite{Keh}).
In particular, when $H$ is open, we have $\rg(H) \subset \rg(G)$.
When $H<G$ is a closed subgroup normalized by a compact subgroup $K<G$,
we equip $\rg(H)$ with the $K$ action induced from the conjugation action $K \curvearrowright H$.

The symbol `$\otimes$' stands for the minimal tensor product of \Cs-algebras,
the Hilbert space tensor product, and the tensor product of unitary representations.
Denote by `$\rc$' the reduced \Cs-crossed product.
(The underlying actions should be always clear from the context.)
For a \Cs-algebra $A$ equipped with a compact group action $K \curvearrowright A$,
denote by $A^K$ the fixed point algebra of the $K$-action.
\subsection*{On conditional expectations}
The following lemma should be well-known for experts.
For completeness of the article, we include the proof.
\begin{Lem}\label{Lem:ce}
Let $K<G$ be a compact open subgroup.
Let $\alpha\colon G \curvearrowright A$ be a \Cs-dynamical system.
Then there is a faithful conditional expectation
\[E_K\colon p_K (A \rc G) p_K \rightarrow p_K A p_K=A^K p_K\]
satisfying
$E_K(p_K a\lambda_s p_K) = \chi_K(s) p_K a p_K$
for all $a\in A$, $s\in G$. 

The analogous statement holds true in the von Neumann algebra setting.
Moreover, in this setting, $E_K$ can be chosen to be normal.
\end{Lem}
\begin{proof}
We only show the \Cs-algebra case.
The proof in the von Neumann algebra case is identical to the \Cs-algebra case.

Take a covariant representation $(\pi, v)$ of $(A, \alpha)$ on $\fH$
such that $\pi$ is faithful.
(For instance, take a faithful regular covariant representation of $(A, \alpha)$.)
We identify
$A \rc G$ with a \Cs-subalgebra of $\IB(\fH \otimes L^2(G))$
via the regular covariant representation associated to $\pi$.
Then this gives rise to an inclusion
$p_K (A \rc G)p_K \subset \IB(\fH \otimes L^2(G)^K)$.
We also identify $p_K A p_K$ with the \Cs-algebra $\pi(A^K)\otimes \id_{\IC \chi_K}$ on $\fH \otimes \IC\chi_K$
in the obvious way.
Let $q$ denote the orthogonal projection
from $\fH \otimes L^2(G)^K$ onto $\fH \otimes \IC\chi_K$.
Define \[\widetilde{E} \colon \IB(\fH \otimes L^2(G)^K) \rightarrow \IB(\fH \otimes \IC \chi_K)\] by \[\widetilde{E}(x):= q x q.\]
Then under the above identifications of \Cs-algebras,
we obtain
\[\widetilde{E}(p_K (A \rc G)p_K) = p_K A p_K.\]
Hence the map $\widetilde{E}$ restricts to a conditional expectation
\[E_K \colon p_K (A \rc G)p_K \rightarrow p_K A p_K.\]
Direct computations show that the map $E_K$ satisfies the required equation.
Let $\rho\colon G \curvearrowright L^2(G)$ denote the right regular representation of $G$.
Observe that $(v\otimes \rho)(G)p_K$
commutes with $p_K(A \rc G) p_K$.
As the subset $[(v \otimes \rho)(G)p_K]\cdot (\fH \otimes \IC \chi_K)$
spans a dense subspace of $\fH \otimes L^2(G)^K$,
the conditional expectation $E_K$ is faithful.
\end{proof}
\section{New construction of non-discrete \Cs-simple groups}

Recall that $G$ is a totally disconnected locally compact group.
We will construct an ambient \Cs-simple group $\cG$ of $G$.

To avoid confusion, we first introduce the following notations.
Let $\Upsilon_n$; $n\in \mathbb{N}$, be pairwise distinct copies of the group
\[\bigoplus_{K<G}\bigoplus_{G/K} \mathbb{Z}_2,\]
where the first direct sum is taken over the set of all compact open subgroups $K$ of $G$.
We equip each $\Upsilon_n$ with the $G$-action induced from the
left translation $G$-actions on $G/K$.
Let $\Xi_n$; $n\in \mathbb{N}$, be pairwise distinct copies of the integer group $\mathbb{Z}$.
We equip each $\Xi_n$ with the trivial $G$-action.

Set \[\Gamma_1:= \Upsilon_1,\quad
\Lambda_1:= \Gamma_1 \ast \Xi_1,\]
equipped with the obvious $G$-actions.
Assume that $\Gamma_n$ and $\Lambda_n$ have been defined.
We then define
\[\Gamma_{n+1}:= \Lambda_n \times \Upsilon_{n+1},\quad
\Lambda_{n+1}:=\Gamma_{n+1} \ast \Xi_{n+1},\]
equipped with the obvious $G$-actions.
As a result, we obtain the increasing sequence
\[\Gamma_1< \Lambda _1 < \Gamma_2 < \Lambda_2 < \cdots\]
of discrete groups.
Define $\Lambda$ to be the inductive limit of the above sequence.
As the inclusions are $G$-equivariant,
we have a natural $G$-action $\alpha$ on $\Lambda$.
Now set
\[\cG:= \Lambda \rtimes_\alpha G.\]
Clearly $\cG$ contains an open subgroup isomorphic to $G$.
Put $\cG_n := \Lambda_n \rtimes G< \cG$ for $n\in \IN$.

For an open compact subgroup $K<G~(< \cG$), the following observation on $p_K{\mathrm C}^\ast_{\mathrm r}(\cG)p_K$ is useful.
Note first that the $\ast$-subalgebra $p_K C_c(\cG) p_K$ is dense in $p_K{\mathrm C}^\ast_{\mathrm r}(\cG) p_K$.
Since $K$ is open in $\cG$, the characteristic functions $\chi_S$; $S \in K\backslash \cG/K$,
form a basis of $p_K C_c(\cG) p_K$.
Therefore one can approximate a given element $x\in p_K{\mathrm C}^\ast_{\mathrm r}(\cG) p_K$ arbitrarily well
by an element of the form
\[\sum_{s\in F} p_K x_s \lambda_s p_K,\]
where $n\in \IN$, $F =\{e, s_1, \ldots, s_l\}$ is a finite subset in $G$
having the pairwise disjoint $K$-double cosets,
$x_s \in \rg(\Lambda_n)$ for all $s\in F$, and $x_e \in \rg(\Lambda_n)^K$.

Note that by Theorem 3.18 of \cite{Wes},
the class of elementary totally disconnected locally compact groups is
closed under taking open subgroups and group extensions.
Therefore the group $\cG$ is elementary if and only if the original group $G$ is elementary.
Typical examples of non-elementary totally disconnected locally compact groups
include $\mathrm{PSL}_d(\mathbb{Q}_p)$ and $\mathrm{Aut}(T_d)$,
where $p$ is a prime number, $d\in \{3, 4, \ldots\}$,
and $T_d$ is a $d$-regular tree (see Proposition 6.3 in \cite{Wes}).

The following theorem is the main result of this article.
\begin{Thm}\label{Thm:simple}
The locally compact group $\cG$ is \Cs-simple.
\end{Thm}
\begin{proof}
Let $I$ be a nonzero (closed two-sided) ideal of $\rg(\cG)= \rg(\Lambda) \rca{\alpha} G$.
Take a nonzero positive element $x\in I$.
Let $K<G$ be a compact open subgroup satisfying $a:=p_K x p_K \neq 0$.
We will show that $p_K \in I$.

Let \[E_K\colon p_K \rg(\cG) p_K \rightarrow \rg(\Lambda)^K p_K\]
be the faithful conditional expectation provided in Lemma \ref{Lem:ce}.
Since $a$ is positive and nonzero, so is $E_K(a)$.
By rescaling $a$ if necessary, we may further assume that
\[\|E_K(a)\|=1.\]
Choose an $n\in \mathbb{N}$ and $a_0\in p_K C_c(\cG_n)p_K$ satisfying
\[\|a - a_0\| <1/2,\quad
\|E_K(a_0)\|=1.\]
Write \[a_0= \sum_{s\in F} p_K x_s \lambda_s p_K;~ x_s\in \rg(\Lambda_n),~ x_e \in \rg(\Lambda_n)^K,\]
where $F=\{e, s_1, \ldots, s_l\}$ is a finite subset of $G$ having
the pairwise distinct $K$-double cosets.
Note that $\| x_e\| = \|E_K(a_0)\|=1$.

Let
$\Upsilon_{n+1, K} < \Upsilon_{n+1}$
be the $K$-th direct summand of $\Upsilon_{n+1}$.
Let $\rg(\Upsilon_{n+1, K}) \cong C(\{0, 1\}^ {G/ K})$ be the obvious $G$-equivariant $\ast$-isomorphism.
Put
\[U:= \left\{(\epsilon_{gK})_{gK\in G/K} \in \{0, 1\}^ {G/ K}:
\epsilon_{K}= 0,~
 \epsilon_{gK}=1 {\rm~for~}gK \subset \bigsqcup_{i=1}^l K s_i K\right\}.\]
Observe that $U$ is a $K$-invariant (nonempty) clopen subset of $\{0, 1\}^ {G/ K}$.
Moreover for each $i$, we have $\alpha_{s_i}(U) \cap U=\emptyset$.
We regard $p:= \chi_U$ as an element of $\rg(\Lambda)$.
Then $p\lambda_{s_i} p=0$ for $i=1, \ldots, l$.
The projection $p$ is nonzero and commutes with $p_K$ and $x_s; s\in F$.
Thus
\[p a_0 p= p x_e p_K.\]
Since $x_e$ and $p$ sit in the first and second tensor product components of $\rg(\Gamma_{n+1})=\rg(\Lambda_n) \otimes \rg(\Upsilon_{n+1})$ respectively, we have
\[\|px_e\| =\|p\| \|x_e\|=1.\]

Let $B$ be the \Cs-subalgebra of $\rg(\Lambda)^K$
generated by $\rg(\Gamma_{n+1})^K$ and $\rg(\Xi_{n+1})$.
By Theorem 2 of \cite{Dy},
$B$ is simple.
Note that $p x_e \in B$.
Therefore, by Lemma 2.3 in \cite{Zac}, one has a sequence $b_1, \ldots, b_r \in B$
satisfying \[\|\sum_{i=1}^r b_i b_i ^\ast \| \leq 2,\quad \sum_{i=1}^r b_i px_e b_i^\ast =1_B.\]
This implies
\[\sum_{i=1}^r b_i p a_0 p b_i^\ast =\sum_{i=1}^r b_i p x_e b_i^\ast p_K=p_K.\]
Since $\|\sum_{i=1}^r b_i p (a- a_0) p b_i^\ast \| \leq 2 \|a- a_0\| <1$, 
we have $\|p_K+ I\|_{\rg(\cG)/I} <1$.
As $p_K$ is a projection, this yields
$p_K \in I$.
Since $K<G$ can be chosen arbitrarily small, we conclude $I=\rg(\cG)$.
\end{proof}

We keep the settings $G, \cG$, and so on until the end of this article.
\section{Uniqueness of KMS weight on $\rg(\cG)$}
By modifying the proof of Theorem \ref{Thm:simple},
we also obtain the uniqueness of KMS weight on $\rg(\cG)$ with respect to the modular flow.
From now on, we freely use the basic facts on the Plancherel weight observed in Section 2.6 of \cite{Rau}.

For a locally compact group $H$,
let $\Delta_H \colon H \rightarrow \IR_{>0}$ denote the modular function of $H$.
Put $H_0 := \ker(\Delta_H)<H$.
Note that for totally disconnected $H$,
it is not hard to see that $H_0$ is open in $H$
and that $\Delta_H(H) \subset \IQ$.
Observe that for our $G$ and $\cG$,
\[\Delta_{G}(G)= \Delta_{\cG}(\cG), \quad \cG_0= \Lambda \rtimes G_0.\]
Let $\varphi$ denote the Plancherel weight on $\rg(\cG)$.
Let $\sigma^\varphi$ be the modular flow on $\rg(\cG)$:
\[[\sigma^\varphi_t(f)](s):= \Delta_{\cG}(s)^{{\mathrm i}t} f(s)\quad{\rm~for~}f\in C_c(\cG),~s\in \cG,~ t\in \IR.\]
Throughout the paper, a weight on a \Cs-algebra
is always assumed to be \emph{densely defined}, \emph{lower semicontinuous}, and nonzero (i.e., \emph{proper}) without stated.
(See Section 2.6 in \cite{Rau} or \cite{KV} for the definitions.)

For a weight $\psi$ on a \Cs-algebra $A$, as in Definition 1.1 of \cite{KV}, denote by $\cM_{\psi}$
the linear span of $\psi^{-1}([0, \infty))$. Note that $\cM_{\psi}$ is a hereditary $\ast$-subalgebra of $A$.
In addition, when $\psi$ is tracial, $\cM_\psi$ is a norm dense ideal of $A$
hence it contains all projections in $A$.
We say the $\ast$-subalgebra
\[\left\{ a\in A: \begin{array}{l}\cM_\psi a \cup a \cM_\psi \subset \cM_{\psi}, \\
\psi(a x) = \psi(xa) {\rm~for~all~} x\in \cM_\psi\end{array}\right\} \subset A\]
the \emph{centralizer} of $\psi$.
We set
\[C_{cc}(\cG):= \bigcup_{K<G} p_K C_c(\cG) p_K.\]
Here the union is taken over all compact open subgroups $K<G$.
Note that $C_{cc}(\cG)$ is a $\ast$-subalgebra of $C_c(\cG)$.
Observe that for any $\sigma^\varphi$-KMS weight $\psi$ on $\rg(\cG)$,
as every $p_K$ is a projection fixed by $\sigma^\varphi$,
we have $C_{cc}(\cG) \subset \cM_\psi$.
Indeed, as $p_K$ is a projection,
one has an analytic element $a\in \cM_\psi$
with $p_K\leq p_K a^\ast a p_K$.
Then, by the KMS condition, we have
\[\psi(p_K) \leq \psi(p_K a^\ast a p_K)= \psi( \sigma^\varphi_{{\mathrm i}/2}(a) p_K\sigma^\varphi_{{\mathrm i}/2}(a^\ast)) \leq \psi( \sigma^\varphi_{{\mathrm i}/2}(a) \sigma^\varphi_{{\mathrm i}/2}(a^\ast))= \psi(a^\ast a).\]

\begin{Thm}\label{Thm:trace}
Up to scalar multiple, the Plancherel weight $\varphi$ is the only
$\sigma^\varphi$-KMS weight on $\rg(\cG)$.
When $\cG$ is non-unimodular, there is no tracial weight on $\rg(\cG)$.
\end{Thm}
\begin{proof}
We consider the following claim.
\begin{claim}
Let $\psi$ be a weight on $\rg(\cG)$ whose centralizer contains $C_{cc}(\cG_0)$
and satisfies $C_{cc}(\cG) \subset \cM_\psi$.
Then for any compact open subgroup $K<G$ and any $s\in \cG \setminus K$, we have
\[\psi(\lambda_s p_K)=0.\]
\end{claim}
Note that by the above observations, any $\sigma^\varphi$-KMS weights and tracial weights on $\rg(\cG)$ satisfy the assumption of the Claim.

We first prove the theorem under the assumption that the Claim holds true.
In the case that $\psi$ is a $\sigma^\varphi$-KMS weight, we will show 
 that $\psi$ is a scalar multiple of $\varphi$.
Take any two compact open subgroups $K_1, K_2 < G$.
Put $K:=K_1 \cap K_2$
and take $s_1, \ldots, s_l, t_1, \ldots, t_r \in G$
satisfying $K_1= \bigsqcup_{i=1}^l s_i K$, $K_2 = \bigsqcup_{i=1}^r t_i K$.
Then
 \[p_{K_1}= \frac{1}{l}\sum_{i=1}^l \lambda_{s_i} p_K,\quad
p_{K_2} = \frac{1}{r}\sum_{i=1}^r \lambda_{t_i} p_K.\]
Since $r \mu(K_1)=l \mu(K_2)$, the hypothesis implies
\[C:=\psi(p_{K_1}) \mu(K_1)= \psi(p_{K_2}) \mu(K_2).\]
By Lemma 2.23 of \cite{Rau} (see also \cite{KV}), 
we obtain $\psi = C\varphi$.
Next consider the case that $\cG$ is non-unimodular and that $\psi$ is a tracial weight.
In this case,
the equality in the Claim implies that the weight $\psi$ vanishes on $C_{cc}(\cG)$.
Since the projections $p_K$, where $K<G$ are compact open subgroups,
form an approximate unit of $\rg(\cG)$,
it follows from the tracial condition and lower semicontinuity of $\psi$
that $\psi=0$.
This proves the statement of the theorem. Hence it suffices to show the Claim.

We now prove the Claim.
Let $\psi$, $K<G$, $s\in \cG \setminus K$ be as in the Claim.
Write $s= g u$; $g\in \Lambda$, $u\in G$.
To show the claimed equation $\psi(\lambda_s p_K)=0$,
we first recall from the proof of Theorem \ref{Thm:simple}
that when $u \not\in K$, one has a nonzero projection $p \in \rg(\Gamma_1)^K$
satisfying $p \lambda_u p=0$.
In fact $p$ is taken from the group algebra $\IC[\Gamma_1]$.
When $u\in K$, put $p:= \lambda_e \in \IC[\Gamma_1]$.
Choose $n\in \IN$ satisfying $g\in \Lambda_n$.
Consider the subgroup $\Sigma:= \Lambda_n \ast \Xi_{n+1} \ast \Xi_{n+2} < \Lambda$.
Denote by $\tau$ the canonical tracial state on $\rg(\Sigma)$.
As observed in Lemma 3.8 of \cite{Suzmin}, thanks to Lemma 5 of \cite{HS}, one can proceed the \emph{Powers averaging argument} \cite{Pow} for $\Sigma$ 
by using only elements in $\Xi_{n+1} \ast \Xi_{n+2}$.
This implies, for any $\varepsilon>0$, one has $t_1, \ldots, t_r \in \Xi_{n+1} \ast \Xi_{n+2}$
satisfying
\[\|\frac{1}{r}\sum_{i=1}^r \lambda_{t_i} x \lambda_{t_i}^\ast - \tau(x)\|< \varepsilon \quad{\rm~for~}x= p, \lambda_g.\]
Then, as $\Xi_{n+1} \ast \Xi_{n+2}$ commutes with $G$, we have
\begin{align*}\|\frac{1}{r^2} \sum_{i=1}^r \sum_{j=1}^r \lambda_{t_i} p \lambda_{t_i}^\ast \lambda_{t_j} p_K \lambda_s p_K \lambda_{t_j}^\ast\lambda_{t_i} p \lambda_{t_i}^\ast\| 
&=\|\frac{1}{r}\sum_{i=1}^r\left[\lambda_{t_i} p \lambda_{t_i}^\ast p_K\left (\frac{1}{r}\sum_{j=1}^r \lambda_{t_j} 
\lambda_g \lambda_{t_j}^\ast\right) \lambda_u p_K \lambda_{t_i} p \lambda_{t_i}^\ast\right]\|\\
&< \varepsilon + \frac{\tau(\lambda_g)}{r}\|\sum_{i=1}^r\lambda_{t_i} p_K p \lambda_u p p_K \lambda_{t_i}^\ast\| \\
&= \varepsilon.\end{align*}
Here the last equation holds true because the condition $u\in K$ implies $g \neq e$.
Since $p$ is a $K$-invariant projection and $\Sigma, K \subset \cG_0$, the previous inequality yields
\begin{align*}|\psi(\lambda_s p_K \sum_{i=1}^r \sum_{j=1}^r \lambda_{t_j}^\ast \lambda_{t_i} p \lambda_{t_i}^\ast \lambda_{t_j} )|
&= |\psi( \sum_{i=1}^r \sum_{j=1}^r\lambda_s p_K (\lambda_{t_j}^\ast \lambda_{t_i} p \lambda_{t_i}^\ast)(\lambda_{t_i} p \lambda_{t_i}^\ast \lambda_{t_j}p_K))|\\
&= |\psi( \sum_{i=1}^r \sum_{j=1}^r \lambda_{t_i} p \lambda_{t_i}^\ast \lambda_{t_j} p_K \lambda_s p_K \lambda_{t_j}^\ast\lambda_{t_i} p \lambda_{t_i}^\ast)|\\
&\leq r^2 \psi(p_K) \varepsilon.
\end{align*}
This yields

\begin{align*}|\tau(p)\psi(\lambda_s p_K)| &\leq |\psi(p_K \lambda_s p_K (\tau(p)- \frac{1}{r^2} \sum_{i=1}^r \sum_{j=1}^r \lambda_{t_j}^\ast \lambda_{t_i} p \lambda_{t_i}^\ast \lambda_{t_j}))| + \psi(p_K)\varepsilon \\ &\leq 2\psi(p_K) \varepsilon.
\end{align*}
Since $\varepsilon>0$ was chosen arbitrarily small (independent on $p$),
we conclude
\[\psi(\lambda_ s p_K)=0.\]
\end{proof}

\section{On Factoriality and types of group von Neumann algebras $L(\cG)$}
In this section we observe the factoriality of the group von Neumann algebra $L(\cG)$.
We then determine its Murray--von Neumann--Connes type.
For the definition of Connes's $S$-invariant,
we refer the reader to Section III of \cite{Con}.

\begin{Thm}
The von Neumann algebra $L(\cG)$ is a non-amenable factor of type
\[
 \begin{cases}
 {\rm I\hspace{-.01em}I}_1 &{\rm~when~}G{\rm~is~discrete},\\
 {\rm I\hspace{-.01em}I}_\infty &{\rm~when~}G{\rm~is~non}\mathchar`-{\rm discrete~and~unimodular},\\

 {\rm I\hspace{-.01em}I\hspace{-.01em}I} &{\rm~ otherwise},
 \end{cases}
\]
whose Connes's $S$-invariant is
the closure of $\Delta_G(G)$ in $\IR_{\geq 0}$.
\end{Thm}
\begin{proof}
We first show that $L(\cG)$ and $L(\cG_0)$ are factors.
By Proposition 2.25 of \cite{Rau}, the centralizer of the Plancherel weight on $L(\cG)$ 
is equal to $L(\cG_0)$.
Hence it suffices to show the factoriality of $L(\cG_0)$.
Let $\varphi$ be the Plancherel weight on $L(\cG_0)$.
Note that $\varphi$ is a faithful normal semifinite tracial weight on $L(\cG_0)$.
Hence if $L(\cG_0)$ is not a factor, then one has a
normal semifinite tracial weight $\psi$ which is dominated by $\varphi$
but is not a scalar multiple of $\varphi$.
This contradicts to the uniqueness of tracial weight on $\rg(\cG_0)$ (up to scalar multiple),
which follows from the proof of Theorem \ref{Thm:trace}.

We next show the non-amenability of $L(\cG)$.
Take any compact open subgroup $K<G$.
Then, by Lemma \ref{Lem:ce}, the corner $p_K L(\cG)p_K$ of $L(\cG)$
admits a conditional expectation
 \[E_K\colon p_K L(\cG) p_K \rightarrow L(\Lambda)^K p_K.\]
Since $L(\Lambda)^K p_K$ is non-amenable,
so is $L(\cG)$.

Finally we determine the Murray--von Neumann--Connes type of $L(\cG)$.
When $G$ is discrete, it is clear from Theorem \ref{Thm:trace} that $L(\cG)$ is of type I\hspace{-.01em}I$_1$.
(Alternatively, in the discrete group case, the statement follows from the fact that $\cG$ is a non-amenable ICC discrete group.)
When $G$ is non-discrete and unimodular, observe that the Plancherel weight on $L(\cG)$
is tracial and unbounded.
Since $L(\cG)$ is non-amenable, it must be of type I\hspace{-.01em}I$_\infty$.
The non-unimodular case and the last statement follow from Connes's theorem \cite{Con}
(see Theorem 2.27 in \cite{Rau}).
\end{proof}

\subsection*{Acknowledgements}
The author is grateful to Sven Raum for
helpful comments on a draft of this article and informing him the reference \cite{Wes}.
He is also grateful to Miho Mukohara for
informing him an inaccuracy in the previous proof of Theorem 4.1.
He would like to thank the referee
for helpful suggestions.
This work was supported by JSPS KAKENHI Early-Career Scientists
(No.~19K14550) and a start-up fund of Hokkaido University.

\end{document}